\documentclass[12pt]{amsart}
\usepackage[utf8]{inputenc}

\usepackage{microtype,amsmath,amsthm,amssymb}
\usepackage{parskip}
\usepackage{geometry}
 \usepackage[hidelinks]{hyperref}
\title{Distribution mod $p$ of Euler's totient and the sum of proper divisors}
\author{Noah Lebowitz-Lockard}
\address{8330 Millman St., Philadelphia, PA 19118}
\email{nlebowi@gmail.com}

\author{Paul Pollack}
\address{Department of Mathematics \\ University of Georgia \\ Athens, GA 30602}
\email{pollack@uga.edu}
\thanks{The second author (P.P.) is supported by NSF award DMS-2001581.}


\author{Akash Singha Roy}
\address{ESIC Staff Quarters No.: D2\\ 143 Sterling Road, Nungambakkam\\Chennai 600034\\ Tamil Nadu, India.}
\email{akash01s.roy@gmail.com}

\renewcommand\phi\varphi

\renewcommand{\pod}[1]{\allowbreak\mathchoice
  {\if@display \mkern 18mu\else \mkern 8mu\fi (#1)}
  {\if@display \mkern 18mu\else \mkern 8mu\fi (#1)}
  {\mkern4mu(#1)}
  {\mkern4mu(#1)}
}
\usepackage{graphicx}
\DeclareMathAlphabet{\curly}{U}{rsfs}{m}{n}
\newcommand\Z{\mathbb{Z}}
\newtheorem{thm}{Theorem}[section]

\newtheorem{prop}[thm]{Proposition}
\newtheorem{lem}[thm]{Lemma}

\theoremstyle{remark}

\begin{document}

\begin{abstract}
We consider the distribution in residue classes modulo primes $p$ of Euler's totient function $\phi(n)$ and the sum-of-proper-divisors function $s(n):=\sigma(n)-n$. We prove that the values $\phi(n)$, for $n\le x$, that are coprime to $p$ are asymptotically uniformly distributed among the $p-1$ coprime residue classes modulo $p$, uniformly for $5 \le p \le (\log{x})^A$ (with $A$ fixed but arbitrary). We also show that the values of $s(n)$, for $n$ composite, are uniformly distributed among all $p$ residue classes modulo every $p\le (\log{x})^A$. These appear to be the first results of their kind where the modulus is allowed to grow substantially with $x$.
\end{abstract}
\maketitle

\section{Introduction}
Let $\phi(n)$ denote \textsf{Euler's totient} and let $s(n)=\sigma(n)-n$ denote the \textsf{sum-of-proper-divisors} (or \textsf{sum-of-aliquot-divisors}) function. In this paper, we determine asymptotic formulas for the number of $n\le x$ for which $\phi(n)$, or $s(n)$, land in a given residue class modulo $p$, uniformly for primes $p$ below any fixed power of $\log{x}$.

For the Euler function, the distribution mod $p$ for fixed $p$ can be read out of known results. Since $\phi(n)$ is even for all $n\ge 3$, one should assume $p$ is odd. Using Wirsing's mean value theorem in \cite{wirsing61}, it is straightforward to prove that the number of $n \le x$ with $\phi(n)$ coprime to $p$ is 
\[ \sim C_p x/(\log{x})^{1/(p-1)},\quad\text{as $x\to\infty$}, \]
for a certain positive constant $C_p$.
(An early reference for this formula is \cite{scourfield64}. See  \cite{SW06} and \cite{FLM14} for more precise results.) In particular, $\phi(n)\equiv 0\pmod{p}$ for $(1+o(1))x$ values of $n\le x$. What about the coprime residue classes? When $p=3$, Dence and Pomerance \cite{DP98} present explicit positive constants $C_{3,1} \approx 0.61$ and $C_{3,2} \approx 0.33$ such that the number of $n\le x$ with $\phi(n)\equiv a\pmod{3}$ is $\sim C_{3,a} x/(\log{x})^{1/2}$, for $a=1,2$. When $p\ge 5$, it follows from work of Narkiewicz (see \cite[Corollary 2]{narkiewicz66} or \cite[Chapter 5]{narkiewicz84}; see also \cite{shirokov83}) that the values of $\phi(n)$ coprime to $p$ are uniformly distributed among the $p-1$ coprime residue classes mod $p$.\footnote{In fact, Narkiewicz shows that if $m$ is coprime to $6$, then the values of $\phi(n)$ that are coprime to $m$ are uniformly distributed among the $\phi(m)$ coprime residue classes modulo $m$.} Hence, for each $a$ coprime to $p$, there are $\sim C_p (p-1)^{-1} x/(\log{x})^{1/(p-1)}$ values of $n\le x$ with $\phi(n)\equiv a\pmod{p}$. 

Our first theorem shows, in a precise form, that uniform distribution over coprime residue classes mod $p$ continues to hold for  $p \le (\log{x})^A$.

\begin{thm}\label{thm:mainphi} Fix $A>0$. Let $x$ and $p$ tend to infinity with $p \le (\log{x})^A$. The number of $n\le x$ with $\phi(n)\equiv a\pmod{p}$ is 
\[ \sim \frac{x}{p (\log{x})^{1/(p-1)}}, \]
uniformly in the choice of coprime residue class $a\bmod{p}$.
\end{thm}

Within its range of validity, Theorem \ref{thm:mainphi} improves  earlier estimates of Banks and Shparlinski (see Theorems 3.1 and 3.2 in \cite{BS04}). 

When $p=o(\log\log x)$, Theorem \ref{thm:mainphi} implies that $p\mid \phi(n)$ for $(1+o(1))x$ values of $n\le x$ (as found already in \cite{EGPS90}; see inequality (4.2) there), while when $p \asymp \log\log x$, Theorem \ref{thm:mainphi} shows that $p\mid \phi(n)$ for $\sim (1-\kappa) x$ values of $n \le x$, where $\kappa = \exp(\frac{-\log\log {x}}{p-1})$. Since $1-\exp(-\log\log x/(p-1)) \sim \log\log{x}/p$ once $\log\log x = o(p)$, it seems reasonable to suspect that $p\mid \phi(n)$ for $\sim x\log\log x/p$ values of $n\le x$ when $p/\log\log{x}\to\infty$. Our next theorem substantiates this, when $p\le (\log{x})^A$.

\begin{thm}\label{thm:secondphi} Fix $A>0$. Suppose that $x$ and $p/\log\log{x}$ tend to infinity, with $p \le (\log{x})^A$. The number of $n\le x$ for which $p\mid \phi(n)$ is $(1+o(1))\frac{x\log\log{x}}{p}$.
\end{thm}

We turn now to $s(n)$. For fixed $p$, one has that $p\mid \sigma(n)$ for all $n$ except those belonging to a set of density $0$. This was observed already by Alaoglu and Erd\H{o}s in 1944 \cite[p.\ 882]{AE44}. (See also the proof of Lemma 5 in \cite{PP13}, and Theorem 2 in \cite{pomerance77}.) Since $s(n)=\sigma(n)-n \equiv -n\pmod{p}$ whenever $p\mid \sigma(n)$, we immediately deduce that $s(n)$ is equidistributed mod $p$ for each fixed $p$. 

We will show that $s(n)$ remains equidistributed for larger $p$, but some care about the formulation is required. Since $s(q)=1$ for every prime $q$, there are at least $(1+o(1))x/\log{x}$ values of $n\le x$ with $s(n)\equiv 1\pmod{p}$, no matter the value of $p$. And this dashes any hope of equidistribution if $p$ is appreciably larger than $\log{x}$. We work around this issue by considering $s(n)$ only for composite $n$.

\begin{thm}\label{thm:mains} Fix $A>0$. As $x\to\infty$, the number of composite $n\le x$ with $s(n)\equiv a\pmod{p}$ is $(1+o(1))x/p$, for every residue class $a$ mod $p$ with $p \le (\log{x})^A$.
\end{thm}

The proofs of Theorem \ref{thm:mainphi} and \ref{thm:mains} combine two different methods. For small $p$, meaning $p$ smaller than roughly $(\log\log x)^2$, we apply the analytic method of Landau--Selberg--Delange. In the (partially overlapping) range when $p$ is a bit larger than $\log\log{x}$, we apply a combinatorial and ``anatomical''\footnote{in the sense of ``anatomy of integers''} method of Banks--Harman--Shparlinski \cite{BHS05}. While similar analytic methods have been applied in such problems before (such as in the work of Narkiewicz mentioned above), the modulus was always fixed. To allow $p$ to grow with $x$, we apply a version of the Landau--Selberg--Delange method enunciated recently by Chang and Martin \cite{CM20}. Interestingly, this part of the argument uses crucially that nontrivial Jacobi sums over $\mathbb{F}_p$ are bounded by $\sqrt{p}$ in absolute value; the trivial bound of $p$ would only allow the method to work for $p$ up to about $\log\log{x}$, just shy of what is required for overlap with our second range.\footnote{For a distinct but related application of these kinds of character sum bounds, see \cite[Chapter 6]{narkiewicz84}. There Weil's bounds are used to prove that certain ``polynomial-like'' multiplicative functions are uniformly distributed in coprime residue classes mod $p$ for all large enough $p$. See also \cite[Theorem 2]{narkiewicz82}.}

Given our reliance on the Siegel--Walfisz theorem, it would seem difficult to extend  uniformity in our results past $(\log{x})^A$. It would be interesting to have heuristics suggesting the ``correct'' range of uniformity to expect. For $s$, uniformity in Theorem \ref{thm:mains} certainly fails as soon as $p$ is a bit larger than $x^{1/2}$. To see this, let $
q,r$ run over primes up to $\frac{1}{3} {\sqrt{x}}$. Then each product $qr \le x$ and $s(qr)=q+r+1 < \sqrt{x}$. Hence, some $m < \sqrt{x}$ has $\gg x^{1/2} (\log{x})^{-2}$  preimages $n=qr\le x$. If now $p \ge x^{1/2} (\log{x})^3$ (say), then the residue class $m\bmod{p}$ contains $s(n)$ for many more than $x/p$ composite $n\le x$.  For $\phi$, a similar argument suggests we should not expect uniformity in Theorem \ref{thm:mainphi} past roughly
\[ L(x) := \exp(\log{x} \cdot \log\log\log{x}/\log\log{x}). \] Indeed, fix $\delta>0$. It was shown by Pomerance --- conditional on a plausible conjecture about shifted primes $q-1$ with no large prime factors --- that for all large $x$, there is an integer $m\le x$ having all prime factors at most $\log{x}$ and possessing at least $x/L(x)^{1+\delta}$ $\phi$-preimages $n\le x$ \cite{pomerance80}. Then if $p \ge L(x)^{1+2\delta}$, the coprime residue class $m\bmod{p}$ contains $\phi(n)$ for many more than $x/p$ values of $n\le x$. 

The reader interested in other work on the distribution of $\phi$ and $s$ in residue classes is referred to \cite{FKP99,BS06,BL07,FS07,BBS08,FL08,BSS09,CG09, garaev09,LPZ11,narkiewicz12,pollack14}.

\subsection*{Notation and conventions} We reserve the letters $p, q, P$ for primes. We write $\log_k$ for the $k$th iterate of the natural logarithm. In addition to employing the Landau--Bachmann--Vinogradov notation from asymptotic analysis, we write $A\gtrsim B$ (resp., $A\lesssim B$) to mean that $A \ge (1+o(1))B$ (resp., $A \le (1+o(1))B$). Constants implied by $O(.)$ or $\ll, \gg$ are absolute unless otherwise  specified.


\section{Preparation}
In this section we collect various results from the literature that will be required in the sequel.  
Let $P^{+}(n)$ denote the largest prime factor of the positive integer $n$, with the convention that $P^{+}(1)=1$. We say that $n$ is \textsf{$Y$-smooth} (or \textsf{$Y$-friable}) if $P^{+}(n)\le Y$.  
For each pair of real numbers $X, Y\ge 1$, we let $$\psi(X,Y) = \#\{n\le X: P^{+}(n)\le Y\},$$  so that $\psi(X,Y)$
gives the count of $Y$-smooth numbers not exceeding $X$. The following estimate is a consequence of the Corollary on p.\ 15 of \cite{CEP83}. 

\begin{lem}\label{lem:smooth} Suppose $X\ge Y \ge 3$, and let $u:=\frac{\log{X}}{\log{Y}}$. Whenever $u\to\infty$ and $X\ge Y \ge (\log{X})^2$, we have
\[ \psi(X,Y) = X \exp(-(1+o(1)) u\log{u}). \]
\end{lem}

The following result is a special case of the fundamental lemma of sieve theory, as formulated  in \cite[Theorem 7.2, p.\ 209]{HR74}.

\begin{lem}\label{lem:sievelem} Let $X \ge Z \ge 3$. Suppose that the interval $I = (u,v]$ has length $v-u=X$. Let $\mathcal{Q}$ be a set of primes not exceeding $Z$. For each $q \in \mathcal{Q}$, choose a residue class $a_q \bmod{q}$.  The number of integers $n\in I$ not congruent to $a_q\bmod{q}$ for any $q \in \mathcal{Q}$ is 
\[ X \left(\prod_{q\in \mathcal{Q}}\left(1-\frac{1}{q}\right)\right)\left(1 + O\left(\exp\left(-\frac{1}{2}\frac{\log{X}}{\log{Z}}\right)\right)\right). \]
\end{lem}

To understand the products over primes appearing in Lemma \ref{lem:sievelem}, we use an estimate due independently to Pomerance (see Remark 1 of \cite{pomerance77}) and Norton (see the Lemma on p.\ 699 of \cite{norton76}).

\begin{lem}\label{lem:PN} Let $m$ be a positive integer and let $a$ be an integer coprime to $m$. Let $p_{a,m}$ denote the least prime $p\equiv a\pmod{m}$. For all $X\ge m$,
\[ \sum_{\substack{p \le X \\ p \equiv a\pmod{m}}} \frac{1}{p} = \frac{\log_2{X}}{\phi(m)} + \frac{1}{p_{a,m}}+ O\left(\frac{\log{(2m)}}{\phi(m)}\right).\]
\end{lem}

\section{Equidistribution of Euler's totient in coprime residue classes: Proof of Theorem \ref{thm:mainphi}}\label{sec:eulerphi}

\begin{lem}\label{lem:phiasymp} Whenever $x,p$, and $\frac{\log{x}}{\log{p}}$ all tend to infinity, we have that
\[ \#\{n\le x: p\nmid \phi(n)\} \sim \frac{x}{(\log{x})^{1/(p-1)}}. \]
\end{lem}
\begin{proof} If $n$ has a prime factor $q\equiv 1\pmod{p}$, then $p\mid \phi(n)$. Now fix a real number $K\ge 1$. If $n \le x$ and $p \nmid \phi(n)$, then $n$ is free of prime factors $q\equiv 1\pmod{p}$, and in particular free of all such prime factors $q\le x^{1/K}$. By Lemma  \ref{lem:sievelem}, the number of such $n\le x$ is \[ x \left(\prod_{\substack{q\equiv 1\pmod{p} \\ q\le x^{1/K}}} \left(1-\frac{1}{q}\right)\right)\left(1 +O\left(\exp\left(-\frac{1}{2}K\right)\right)\right).\]
Moreover,
\[ \prod_{\substack{q\equiv 1\pmod{p} \\ q\le x^{1/K}}} \left(1-\frac{1}{q}\right) = \exp\left(-\sum_{\substack{q\equiv 1\pmod{p} \\q \le x^{1/K}}} \left(\frac{1}{q} + O(1/q^2)\right)\right). \]
Since $q>p$ for every $q\equiv 1\pmod{p}$, the sum of the $O(1/q^2)$ terms will be $O(1/p)$. Also, from Lemma \ref{lem:PN}, once $x, p$, and $\frac{\log{x}}{\log{p}}$ are large enough (possibly depending on $K$),
\[ \sum_{\substack{q \equiv 1\pmod{p}\\ q\le x^{1/K}}} \frac{1}{q} = \frac{\log_2{x}}{p-1} + O\left(\frac{\log{K}}{p} + \frac{\log{p}}{p}\right). \]
Putting these estimates back in above, we find that the count of $n\le x$ with $p \nmid \phi(n)$ is (for large $x,p, \frac{\log{x}}{\log{p}}$) at most 
\[ \frac{x}{(\log{x})^{1/(p-1)}} \left(1+O\left(\frac{\log{K}}{p} + \frac{\log{p}}{p} + \exp\left(-\frac{1}{2}K\right)\right)\right),\]
which is (for large $p$) at most $(1+O(\exp(-K/2))) x/(\log{x})^{1/(p-1)}$. Since $K$ can be taken arbitrarily large, the upper bound half of Lemma \ref{lem:phiasymp} follows.

The lower bound is similar. Again, fix $K \ge 1$. From our earlier work, the count of $n\le x$ having no prime factor $q\equiv1\pmod{p}$ with $q\le x^{1/K}$ is (for large $x,p, \frac{\log{x}}{\log{p}}$) $(1+O(\exp(-K/2))) x/(\log{x})^{1/(p-1)}$. Moreover, the same estimate holds if require also that $p\nmid n$. (We acquire an extra factor of $(1-1
/p)$ in our sieve argument, which can be absorbed into $(1+O(\exp(-K/2)))$ for large $p$.) 

Suppose that $n\le  x$ is coprime to $p$ and free of primes $q\equiv 1\pmod{p}$ with $q\le x^{1/K}$ but that nevertheless  $p\mid \phi(n)$. Write $n=AB$, where $A$ is the largest divisor of $n$ composed of primes $q\equiv 1\pmod{p}$. We count the number of $A$ corresponding to a given $B$. Observe that $1< A \le x/B$ and that every prime dividing $A$ exceeds $x^{1/K}$. Also, $A\equiv 1\pmod{p}$, and so $A=1+pa$ for some $a < x/pB$. We can assume that $\frac{\log{x}}{\log{p}} > 2K$, so that $x/pB = (x/B)/p \ge A/p > x^{1/K}/x^{1/2K} = x^{1/2K}$. So by Lemma \ref{lem:sievelem} (sieving $a$, with primes up to $x^{1/2K}$), the number of $A$ corresponding to a given $B$ is \begin{equation}\label{eq:numAB} \ll \frac{x}{pB}\prod_{q \le x^{1/2K},~q\ne p} \left(1-\frac{1}{q}\right) \ll \frac{Kx}{pB\log{x}}.\end{equation}
Since $B$ is free of prime factors $q\equiv 1\pmod{p}$, Mertens' theorem yields \begin{align*}  \sum \frac1B \le \prod_{\substack{q\not\equiv 1\pmod{p}\\q\le x}}(1-1/q)^{-1}&\ll (\log{x})\prod_{\substack{q\equiv 1\pmod{p}\\q\le x}}(1-1/q) \\ 
 &\ll (\log{x}) \exp\Bigg(-\sum_{\substack{q\equiv 1\pmod{p} \\q\le x}}\frac{1}{q}\Bigg) \ll (\log{x})^{1-\frac{1}{p-1}}, \end{align*}
using Lemma \ref{lem:PN} in the last step. Hence, the number of these $n$ is $O(\frac{K}{p} x/(\log{x})^{1/(p-1)})$, which is $O(\exp(-K/2) x/(\log{x})^{1/(p-1)})$ for large $p$. 

From our last two paragraphs, the count of $n\le x$ for which $p\nmid \phi(n)$ is at least $(1+O(\exp(-K/2))) x/(\log{x})^{1/(p-1)}$, for large $x$, $p$, and $\frac{\log{x}}{\log{p}}$. Taking $K$ large completes the proof of the lower bound.
\end{proof}

Using Lemma \ref{lem:phiasymp} and the method of Landau--Selberg--Delange, we can prove Theorem \ref{thm:mainphi} in the range $p \le (\log_2{x})^{2-\delta}$.
\begin{lem}\label{lem:SDphi} Fix $\delta > 0$. Suppose that $x,p \to \infty$, with $p \le (\log_2 x)^{2-\delta}$. The number of $n\le x$ with $\phi(n)\equiv a\pmod{p}$ is
\[ \sim \frac{x}{p (\log{x})^{1/(p-1)}}, \]
uniformly in the choice of $a$ coprime to $p$.
\end{lem}
We defer the proof of Lemma \ref{lem:SDphi} to \S\ref{sec:SD}. 

Suppose that $x,p\to\infty$ in such a way that $p/\log_2{x}\to\infty$. Then $(\log{x})^{1/(p-1)} \sim 1$, and  $x/(\log{x})^{1/(p-1)} \sim x$. Thus, to finish off Theorem \ref{thm:mainphi}, it will suffice to establish the next two propositions. 

\begin{prop}\label{prop:philargecaseupper} Fix $A>0$. The number of $n\le x$ for which $\phi(n)\equiv a\pmod{p}$ is $\lesssim x/p$ as $x, p \to \infty$, uniformly in $a,p$ with $p \le (\log{x})^A$ and $a\in \Z$ coprime to $p$.
\end{prop}

\begin{prop}\label{prop:philargecaselower} Fix $A>0$. The number of $n\le x$ for which $\phi(n)\equiv a\pmod{p}$ is $\gtrsim x/p$ as $x, \frac{p}{\log_2 x} \to \infty$, uniformly in $a,p$ with $p \le (\log{x})^A$ and $a\in \Z$ coprime to $p$.
\end{prop}

The proofs of Propositions \ref{prop:philargecaseupper} and \ref{prop:philargecaselower} both begin the same way. In what follows, we assume $x,p\to\infty$ and that $p\le (\log{x})^A$, for a fixed $A>0$. We set $L:=\exp(\sqrt{\log x})$.

For each $n>1$, we may think of $n$ as factored in the form $n=mP$, where $P = P^{+}(n)$. Then
\[ \sum_{\substack{1 < n\le x\\ \phi(n)\equiv a\pmod{p}}} 1 = \sum_{\substack{m, P:~ mP\le x\\ P \ge P^{+}(m) \\ \phi(Pm) \equiv a\pmod{p}}}1. \]
By Lemma \ref{lem:smooth}, the number of $n\le x$ for which $P \le L$ is $O(x/L)$, which is $o(x/p)$ in our range of $p$. Such a contribution is negligible from the point of view of our asymptotic formulas. Thus, we may assume that $P > L$. We can also assume $P^2\nmid Pm=:n$ (equivalently, $P> P^{+}(m)$), since the number of $n\le x$ divisible by $r^2$ for an integer $r > L$ is at most $x\sum_{r > L} r^{-2} \ll x/L$. Then $\phi(Pm)=(P-1)\phi(m)$. For a given $m$, the congruence $(P-1) \phi(m) \equiv a\pmod{p}$ holds for all $P$ in a certain coprime residue class $a_{p,m}\bmod{p}$ as long as $p\nmid \phi(m)$ and $\phi(m)\not\equiv -a\pmod{p}$. So writing $L_m := \max\{P^{+}(m),L\}$, 
\begin{equation}\label{eq:toestimate} \sum_{\substack{m, P:~mP\le x\\ P \ge P^{+}(m) \\ \phi(Pm) \equiv a\pmod{p}}}1 = \Bigg(\sum_{\substack{m \le x \\ \phi(m)\not\equiv 0,-a\pmod{p} \\ L_m < x/m}} \sum_{\substack{L_m < P \le x/m \\ P \equiv a_{p,m}\pmod{p}}} 1\Bigg) + o(x/p). \end{equation}

Since $p \le (\log{x})^{A} \le (\log(L_m))^{2A}$, the Siegel--Walfisz theorem (see \cite[Corollary 11.21]{MV07}) implies that, for a certain absolute positive constant $c$,
\begin{equation}\label{eq:toplug} \sum_{\substack{L_m < P \le x/m \\ P \equiv a_{p,m}\pmod{p}}} 1 = \frac{1}{p-1} \sum_{\substack{L_m < P \le x/m}} 1 + O_A\left(\frac{x}{m} \exp(-c\sqrt{\log(x/m)})\right). \end{equation}
Since $\log(x/m)^{1/2} \ge (\log{x})^{1/4}$, 
if we plug \eqref{eq:toplug} into the right-hand side of \eqref{eq:toestimate}, the $O$-terms contribute
\[ \ll_{A} x \exp(-c(\log{x})^{1/4}) \sum_{m \le x} \frac{1}{m} \ll x \exp\left(-\frac{1}{2}c (\log{x})^{1/4}\right), \]
which is $o(x/p)$. The main terms contribute
\[ \frac{1}{p-1}\sum_{\substack{m \le x \\ \phi(m)\not\equiv 0,-a\pmod{p} \\ L_m < x/m}}  \sum_{\substack{L_m < P \le x/m}}1.\]

Carrying out our earlier simplifications, but in reverse, we find that
\[ \sum_{\substack{m \le x \\ \phi(m)\not\equiv 0,-a\pmod{p} \\ L_m < x/m}}  \sum_{\substack{L_m < P \le x/m}}1 = \Bigg(\sum_{\substack{m, P:~mP\le x \\ P \ge P^{+}(m) \\ \phi(m)\not\equiv 0,-a\pmod{p}}} 1\Bigg) + o(x/p). \] 
Putting all of this together yields following fundamental relation:
\begin{equation}\label{eq:fundamentalphi} \sum_{\substack{1 < n \le x \\ \phi(n) \equiv a\pmod{p}}} 1 = \Bigg(\frac{1}{p-1} \sum_{\substack{m, P:~mP\le x \\ P \ge P^{+}(m) \\ \phi(m)\not\equiv 0,-a\pmod{p}}} 1\Bigg) + o(x/p).\end{equation}

\begin{proof}[Proof of  Proposition \ref{prop:philargecaseupper}] 
The right-hand sum in \eqref{eq:fundamentalphi} is trivially bounded by $x$, since every integer $n>1$ has a unique representation in the form $mP$ with $P \ge P^{+}(m)$. Hence, the right-hand side of \eqref{eq:fundamentalphi} is at most $x/(p-1) + o(x/p) = (1+o(1))x/p$, as desired.
\end{proof}

\begin{proof}[Proof of Proposition \ref{prop:philargecaselower}] Since $\sum_{\substack{m, P:~mP\le x \\ P \ge P^{+}(m)}} 1 = x + O(1)$, in view of \eqref{eq:fundamentalphi} the claimed lower bound will follow if it is shown that both
\begin{equation}\label{eq:phi1} \sum_{\substack{m, P:~mP\le x \\ P \ge P^{+}(m) \\ \phi(m)\equiv 0\pmod{p}}} 1 = o(x)\end{equation}
and
\begin{equation}\label{eq:phi2} \sum_{\substack{m, P:~mP\le x \\ P \ge P^{+}(m) \\ \phi(m)\equiv -a\pmod{p}}} 1 = o(x). \end{equation}

If $n=mP$ is counted by the left-hand side of \eqref{eq:phi1}, then $n\le x$ and $p\mid \phi(m) \mid \phi(n)$. Since $p/\log_2{x}\to\infty$, Lemma \ref{lem:phiasymp} puts $n$ in a set of size $o(x)$, proving \eqref{eq:phi1}.


We turn now to \eqref{eq:phi2}. We first consider all $n$ with $1 < n\le x$ of the form $n=mP$, $P \ge P^{+}(m)$, having $m \le L:=\exp(\sqrt{\log{x}})$. The number of such $n$ does not exceed
\[ \sum_{m \le L} \sum_{P \le x/m} 1 \ll \frac{x}{\log{x}} \sum_{m \le L} \frac{1}{m} \ll \frac{x}{\sqrt{\log x}}= o(x).\]
So for the purpose of establishing \eqref{eq:phi2}, we may tack on to its left-hand side the condition that $m > L$. Then $x/P > L$. We now bound the number of $n=mP$ that occur by counting, for each $P$, the number of corresponding $m\le x/P$. Since $p \le (\log{x})^{A}<  (\log(x/P))^{2A}$, we may apply Proposition \ref{prop:philargecaseupper}.  We find that if $p$ and $x$ are sufficiently large and in our given range,
\[ \sum_{\substack{m, P:~mP\le x \\ P \ge P^{+}(m) \\ m \ge L\\ \phi(m)\equiv -a\pmod{p}}} 1 \le \sum_{P\le x/L} \sum_{\substack{m \le x/P \\ \phi(m) \equiv -a\pmod{p}}} 1 \le \frac{2x}{p}\sum_{P\le x}\frac{1}{P}, \]
which is $\ll x\log_2{x}/p = o(x)$, as desired. \end{proof}

\section{Values of $\phi(n)$ divisible by $p$: Proof of Theorem \ref{thm:secondphi}}

We suppose, as in the statement of Theorem \ref{thm:secondphi}, that $x$ and $p/\log_2 x$ tend to infinity, with $p\le (\log{x})^A$. We start the proof by showing that $\sum_{q\le x,~q\equiv 1\pmod{p}} 1/q \sim \log_2 x/p$. For this  we adapt Pomerance's proof of  Lemma \ref{lem:PN}. Fix $K\ge A$. Noting that any prime congruent to $1$ mod $p$ exceeds $p$, we see that
\[ \sum_{\substack{q \le x \\ q\equiv 1\pmod{p}}}\frac{1}{q} = O(1/p) + \int_{10p}^{\exp(p^{1/K})} \frac{\mathrm{d}\pi(t;p,1)}{t} + \int_{\exp(p^{1/K})}^{x} \frac{\mathrm{d}\pi(t;p,1)}{t}.\]
We  assume throughout this argument that $x$ and $p/\log_2 x$ are large (allowed to depend on $K$). Then $10p < \exp(p^{1/K})$. By the Brun--Titchmarsh inequality (see, e.g., \cite[Theorem 3.9, p.\ 90]{MV07}), $\pi(t;p,1) \ll \frac{t}{p\log(t/p)}$ for all $t>p$, and so the first right-hand integral in the last display is 
\[ \ll \frac1p + \frac1p\int_{10p}^{\exp(p^{1/K})} \frac{\mathrm{d}t}{t\log(t/p)} \ll \frac{\log{p}}{Kp}. \]
By the Siegel--Walfisz theorem, for all $t\ge \exp(p^{1/K})$, \begin{align*} \pi(t;p,1) &= \frac{\mathrm{li}(t)}{p-1}  + O_K(t\exp(-c\sqrt{\log t}))\\
&= \frac{t}{(p-1)\log{t}} + O\left(\frac{t}{p (\log{t})^2}\right) + O_K(t\exp(-c\sqrt{\log t})),
\end{align*}
leading to the conclusion that 
\begin{align*}
\int_{\exp(p^{1/K})}^{x} \frac{\mathrm{d}\pi(t;p,1)}{t}  &=\frac{\log_2{x}}{p-1} + O\left(\frac{\log{p}}{Kp}\right) + O_K\left(\frac{1}{p}\right) \\ &= \frac{\log_2{x}}{p} + O\left(\frac{\log_2{x}}{p^2} + \frac{\log{p}}{Kp}\right) + O_K\left(\frac{1}{p}\right). \end{align*}
Assembling these estimates, we find that if $x, p/\log_2{x}$ are large and $p\le (\log{x})^A$, 
\[ \sum_{\substack{q\le x\\q\equiv 1\pmod{p}}}\frac{1}{q} = \frac{\log_2{x}}{p} \left(1+ O(A/K)\right).  \]
Since $K$ can be taken arbitrarily large, $\sum_{q\le x,~q\equiv 1\pmod{p}} 1/q \sim \log_2{x}/p$, as claimed.

The upper bound in Theorem \ref{thm:secondphi} now follows quickly. If $p \mid \phi(n)$, either $p^2\mid n$ or $q\mid n$ for some $q\equiv 1\pmod{p}$. The former occurs for at most $x/p^2$ values of $n\le x$, which is negligible compared to $x\log_2{x}/p$. The latter occurs for at most $x\sum_{q\le x,~q\equiv1\pmod{p}}1/q = (1+o(1)) x\log_2{x}/p$ values of $n$. 

For a lower bound, it is enough to bound from below the number of $n\le x$ having at least one prime factor $q\equiv 1\pmod{p}$. We perform the first two steps of inclusion-exclusion. Let $N_1$ count each $n\le x$ weighted by $k(n)$, where $k(n)$ is the number of its distinct prime divisors $q\equiv 1\pmod{p}$, and let $N_2$ count each $n\le x$ weighted by $\binom{k(n)}{2}$. Since $k - \binom{k}{2} \le 1$ for each integer $k \ge 0$, our count is bounded below by $N_1-N_2$. Now $N_1 = \sum_{q\le x,~q\equiv 1\pmod{p}}\lfloor x/q\rfloor = (x\sum_{q\le x,~q\equiv 1\pmod{p}} 1/q) + O(x/p\log{x}) = (1+o(1)) x\log_2{x}/p$, while
\[ N_2 = \sum_{\substack{q_1 < q_2 \le x \\ q_1\equiv q_2\equiv 1\pmod{p}}} \left\lfloor\frac{x}{q_1 q_2} \right\rfloor \le x \Bigg(\sum_{\substack{q\le x \\ q\equiv 1\pmod{p}}}\frac{1}{q}\Bigg)^2 = (1+o(1)) \frac{x(\log_2{x})^2}{p^2}, \]
which is $o(x\log_2 x/p)$.

\section{Equidistribution of the sum of proper divisors: Proof of Theorem \ref{thm:mains}} 

As explained in the introduction, we may confine our attention to the situation when $p\to\infty$.

\begin{lem}\label{lem:sigmaasymp} Fix $A>0$. Suppose that $p, x, \frac{\log{x}}{\log{p}}\to\infty$. Then, uniformly in the choice of residue class $a\bmod{p}$, \[ \sum_{\substack{n \le x \\ n \equiv a\pmod{p} \\ \sigma(n)\not\equiv 0\pmod{p}}}1 \sim \frac{x}{p (\log{x})^{1/(p-1)}}. \]
\end{lem}

\begin{proof} The proof is similar to that of Lemma \ref{lem:phiasymp}. First we treat the upper bound. Suppose that $n \le x$, $n\equiv a\pmod{p}$, and $\sigma(n)\not\equiv 0\pmod{p}$. Write $n=AB$, where $A$ is the largest divisor of $n$ composed of primes congruent to $-1\pmod{p}$. Then $A$ is squarefull, $A\equiv \pm 1\pmod{p}$, and $B\equiv \pm a\pmod{p}$ (with matching choices of sign). The number of $n\le x$ with a squarefull divisor exceeding $x^{1/2}$ is at most $x \sum_{m > x^{1/2},~\text{squarefull}} 1/m \ll x^{3/4}$, which is $o(\frac{x}{p (\log{x})^{1/(p-1)}})$ as $x,p, \frac{\log{x}}{\log{p}}$ tend to infinity. So we assume that $A \le x^{1/2}$ and count $B$  corresponding to a given $A$. We have that $B\le x/A$, that $B\equiv \pm a\pmod{p}$ (for a specific choice of sign, determined by $A$), and that $B$ is free of prime factors $q\equiv -1\pmod{p}$. In particular, fixing $K\ge 4$, we have that $B$ is free of prime factors $q\equiv -1\pmod{p}$ with $q\le x^{1/K}$. Since $x^{1/K} \le x^{1/4} \le \frac{x}{Ap}$ when $\frac{\log{x}}{\log{p}}\ge 4$, the sieve bounds the number of these $B$ by
\[ \left(\frac{x}{Ap} \prod_{\substack{q \le x^{1/K}\\q\equiv -1\pmod{p}}}(1-1/q)\right) 
\left(
1+O\left(
\exp\left(
-\frac{1}{2} \frac{\log(x/Ap)}{\log(x^{1/K})}
\right)
\right)
\right),
\]
which (cf.\ the proof of Lemma \ref{lem:phiasymp}) is at most
\[ \frac{x}{Ap(\log{x})^{1/(p-1)}} \left(1+O(\exp(-K/8))\right) \]
when $x,p, \frac{\log{x}}{\log{p}}$ are all large enough (allowed to depend on $K$). The sum of $1/A$ over squarefull positive integers $A\equiv \pm 1\pmod{p}$ is at most $1+\sum_{A\ge p-1,~A\text{ squarefull}}1/A = 1+O(p^{-1/2})$, which is $1+O(\exp(-K/8))$ for large $p$. The upper bound half of the lemma now follows, since $K$ can be taken arbitrarily large.

We start the proof of the lower bound by counting $n \le x$, $n\equiv a\pmod{p}$ with no small prime factor $q\equiv -1\pmod{p}$. Taking ``small'' to mean $q \le x^{1/K}$, where $K\ge 2$ is fixed, the sieve implies that the number of such $n\le x$, when $x$, $p$, and $\frac{\log{x}}{\log{p}}$ are all large, is
\begin{multline*} \Bigg(\frac{x}{p}\prod_{\substack{q\equiv -1\pmod{p} \\ q\le x^{1/K}}}\left(1-\frac{1}{q}\right)\Bigg) \left(1 + O\left(\exp\left(-\frac{1}{2}\frac{\log{(x/p)}}{\log(x^{1/K})}\right)\right)\right)
\\
= \frac{x}{p (\log{x})^{1/(p-1)}} \left(1+O(\exp(-K/4))\right).  \end{multline*}
We now wish to remove from our count those $n$ that survive the sieve of the last paragraph but nonetheless satisfy $\sigma(n)\equiv 0\pmod{p}$. Take an $n$ of this kind. We consider two cases, according to whether or not there is a prime $q$ dividing $n$ with $q\equiv -1\pmod{p}$.

Suppose there is such a prime $q$. Since $n$ survived our sieve, necessarily $q> x^{1/K}$. Let $A$ be the largest divisor of $n$ composed of primes $q\equiv -1\pmod{p}$ and write $n=AB$. Then $A\equiv \pm 1\pmod{p}$ and $B\equiv \pm a\pmod{p}$  (for the same choice of sign). As in the proof of Lemma \ref{lem:phiasymp} (see \eqref{eq:numAB}), the number of $A$ corresponding to a given $B$ is 
$$\ll \frac{Kx}{pB\log{x}}.$$ 
(As usual, we assume all of $x,p,\frac{\log{x}}{\log{p}}$ are large.) We now estimate $\sum 1/B$. For each $T\ge p^2$, the sieve (along with Lemma \ref{lem:PN}) implies that the number of $B \le T$, $B\equiv \pm a\pmod{p}$, with $B$ free of prime factors $q\equiv -1\pmod{p}$ is 
\[ \ll \frac{T/p}{\log(T/p)^{1/(p-1)}} \ll \frac{T}{p (\log{T})^{1/(p-1)}}. \]
Summing by parts, 
\[ \sum \frac{1}{B} \ll 1+ \frac{1}{p} (\log{x})^{1-\frac{1}{p-1}}. \]
(The ``$1$'' bounds the contribution of those $B \le p^2$.) 
Hence, the count of corresponding $n$ is $$ \ll \frac{Kx}{p\log{x}}+\frac{Kx}{p^2 (\log{x})^{1/(p-1)}}, $$
which is $o(\frac{x}{p (\log{x})^{1/(p-1)}})$ as $x, p, \frac{\log{x}}{\log{p}}$ tend to infinity.

Now suppose that $n$ is entirely free of primes $q\equiv -1\pmod{p}$. In that case, since $p\mid \sigma(n)$, there must be a prime power $q^e\parallel n$, $e>1$, for which $p \mid\sigma(q^e)$. Let $S$  be the product of all such $q^e\parallel n$. If $S\ge x^{1/2}$, then $S$ is a squarefull divisor of $n$ exceeding $x^{1/2}$; as at the start of this proof, this puts $n$ in a set of size $O(x^{3/4})$, which is $o(\frac{x}{p (\log{x})^{1/(p-1)}})$. So suppose that $S \le x^{1/2}$ and write $n=ST$. Then $T \le x/S$, $T \equiv aS^{-1}\pmod{p}$, and $T$ is free of primes $q\equiv -1\pmod{p}$. By another application of the sieve, the number of possibilities for $T$  given $S$ is 
\[ \ll \frac{x}{pS} \prod_{\substack{q\equiv -1\pmod{p} \\ q\le \frac{x}{pS}}}\left(1-\frac{1}{q}\right) \ll \frac{x}{pS (\log(x/pS))^{1/(p-1)}}\ll \frac{x}{pS (\log{x})^{1/(p-1)}}, \]
when $x, p, \frac{\log{x}}{\log{p}}$ are all large. To estimate $\sum 1/S$, note that $\sigma(q^e) < 2q^e$ for every prime power $q^e$, so that if $p\mid\sigma(q^e)$, then $q^e > \frac{1}{2}p$. It follows that $\sum 1/S \le \sum_{\substack{S > \frac{1}{2}p,~S\text{ squarefull}}} 1/S \ll 1/p^{1/2}$. So only $O(\frac{x}{p^{3/2} (\log{x})^{1/(p-1)}})$ values of $n$ arise this way, and this is $o(\frac{x}{p (\log{x})^{1/(p-1)}})$. 

The lower bound half of the lemma follows by combining the results of the previous three paragraphs, noting again that $K$ can be as large as we like.
\end{proof}

\subsection{Equidistribution when $p\le (\log_2{x})^{2-\delta}$}
The proof of the next lemma, concerning the joint distribution of $n$ and $\sigma(n)$ mod $p$, is deferred to \S\ref{sec:SD}.
\begin{lem}\label{lem:SDs} Fix $\delta > 0$. Suppose that $p, x\to \infty$, with $p \le (\log_2 x)^{2-\delta}$. The number of $n\le x$ with $n\equiv u\pmod{p}$ and $\sigma(n)\equiv v\pmod{p}$ is 
\[ \sim \frac{x}{p^2 (\log{x})^{1/(p-1)}}, \]
uniformly in the choice of integers $u,v$ coprime to $p$.
\end{lem}

With Lemmas \ref{lem:sigmaasymp} and \ref{lem:SDs} in hand, we can deduce Theorem \ref{thm:mains} in the range $p \le (\log_2 x)^{2-\delta}$ ($\delta > 0$ fixed). Notice that in this range, it makes no difference if we restrict the inputs of $s(\cdot)$ to composite $n$, since $x/\log{x}=o(x/p)$. 

We can express the count of $n\le x$ with $s(n)\equiv a\pmod{p}$ as \begin{equation}\label{eq:uvsplit} \sum_{\substack{u,v\pmod{p}\\u+v\equiv a\pmod{p}}} N_{u,v;\, p}(x),\end{equation}
where
\[ N_{u,v;\,p}(x):= \sum_{\substack{n \le x \\ n\equiv -u\pmod{p} \\ \sigma(n)\equiv v\pmod{p}}} 1. \]

First, suppose that $a\not\equiv 0\pmod{p}$. Then there are $p-2$ pairs $(u,v)$ summing to $a$ mod $p$ with $u,v\not\equiv 0 \pmod{p}$. By Lemma \ref{lem:SDs}, $N_{u,v;\,p}(x)  \sim \frac{x}{p^2 (\log{x})^{1/(p-1)}}$ for each, resulting in a combined contribution to \eqref{eq:uvsplit} of $(1+o(1)) \frac{x}{p(\log{x})^{1/(p-1)}}$. The two remaining pairs are $(0,a)$ and $(a,0)$. Suppose $n$ is counted by $N_{0,a;\, p}(x)$. Write $n=p k$. Then $\sigma(k)\equiv \sigma(n)\equiv a\pmod{p}$. Now taking cases according to whether $p\nmid k$ or $p\mid k$, and writing $k=pk'$ in the latter, we find that
\[ N_{0,a;\, p}(x) \le \sum_{\substack{k \le x/p \\ k\not\equiv 0\pmod{p} \\ \sigma(k)\equiv a\pmod{p}}}1  + \sum_{\substack{k' \le x/p^2 \\ \sigma(k')\not\equiv0\pmod{p}}}1. \]
Here the first sum can be estimated by Lemma \ref{lem:SDs} while the second succumbs to Lemma \ref{lem:sigmaasymp}; the sums total to $o(\frac{x}{p(\log{x})^{1/(p-1)}})$. A further application of Lemma \ref{lem:SDs} shows that $$ N_{a,0;\, p}(x) = \frac{x}{p} - (1+o(1))\frac{x}{p(\log{x})^{1/(p-1)}}.$$ 
Combining our tallies, the $n$ with $s(n)\equiv a\pmod{p}$ make up a set of size $x/p + o(\frac{x}{p(\log{x})^{1/(p-1)}})$, which is $(1+o(1))x/p$, as desired.

The argument is similar when $a\equiv 0\pmod{p}$. In that case, there are $p-1$ contributions of size $(1+o(1))  \frac{x}{p^2 (\log{x})^{1/(p-1)}}$ coming from the pairs $(u,-u)$ with $u\not\equiv 0\pmod{p}$, for a total of $(1+o(1))  \frac{x}{p (\log{x})^{1/(p-1)}}$. It remains to consider $N_{0,0;\, p}(x)$. Writing the integers $n$ counted by $N_{0,0;\, p}(x)$ in the form $p^r k$, where $p\nmid k$, we see using Lemma \ref{lem:sigmaasymp} that
\begin{align*} N_{0,0;\, p}(x) &= \sum_{\substack{k \le x/p\\ k\not\equiv 0\pmod{p} \\ \sigma(k)\equiv 0\pmod{p}}} 1 + O(x/p^2) \\
&= (p-1)\left(\frac{x}{p^2} - (1+o(1)) \frac{x}{p^2 (\log{x})^{1/(p-1)}}\right) + O(x/p^2) \\
&= x/p - (1+o(1)) \frac{x}{p (\log{x})^{1/(p-1)}} + O(x/p^2).
\end{align*}
Tallying it all up, we get a total of $(1+o(1))x/p$ in this case as well. This completes the proof of Theorem \ref{thm:mains} when $p \le (\log_2 x)^{2-\delta}$. 

\subsection{Equidistribution when $p/\log_2{x}\to\infty$}
For the remainder of this section, we work in the range where both $x$ and $p/\log_2{x}$ tend to infinity. We continue to assume that $p \le (\log{x})^A$, where $A>0$ is fixed. 

Suppose $n$ is composite with $1<n\le x$ and write $n=mP$ where $P=P^{+}(n)$. Set $L:=\exp(\sqrt{\log x})$. As in \S\ref{sec:eulerphi}, we can assume that $P > L$ and $P\nmid m$, at the cost of $o(x/p)$ exceptions. Then $s(n) = (P+1)\sigma(m)-Pm = Ps(m)+\sigma(m)$, and we have $s(n)\equiv a\pmod{p}$ precisely when $Ps(m)\equiv a-\sigma(m)\pmod{p}$. Now writing $L_m = \max\{L,P^{+}(m)\}$, we see that
\[ \sum_{\substack{1 < n \le x\\ n~\text{composite} \\ s(n)\equiv a\pmod{p}}} 1 = \Bigg(\sum_{\substack{1 < m \le x \\ s(m)\equiv 0 \pmod{p}\\\sigma(m)\equiv a\pmod{p}}} \sum_{L_m <P \le x/m}1 + \sum_{\substack{1 < m \le x \\ s(m)\not\equiv 0 \pmod{p}\\\sigma(m)\not\equiv a\pmod{p}}} \sum_{\substack{L_m <P \le x/m \\ P \equiv a_{p,m} \pmod{p}}} 1\Bigg) + o(x/p),  \]
where $a_{p,m}\bmod{p}$ is  determined by the congruence $a_{p,m} \cdot s(m)\equiv a-\sigma(m)\pmod{p}$. Proceeding in exact analogy with \S\ref{sec:eulerphi}, we may express the right-hand side as
\begin{equation}\label{eq:exactanalogy} \Bigg(\sum_{\substack{m,P:~mP\le x\\ m>1,~P \ge P^{+}(m) \\ s(m)\equiv 0 \pmod{p}\\ \sigma(m)\equiv a\pmod{p}}} 1 + \frac{1}{p-1} \sum_{\substack{m,P:~mP\le x\\ m>1,~P \ge P^{+}(m) \\ s(m)\not\equiv 0 \pmod{p}\\ \sigma(m)\not\equiv a\pmod{p}}} 1\Bigg)+o(x/p). \end{equation}

We proceed to show that the first of the two sums in \eqref{eq:exactanalogy} is $o(x/p)$. 

Take first the case when $p \mid a$. If $m,P$ are counted by this first sum, then $m=\sigma(m)-s(m) \equiv a-0\equiv 0\pmod{p}$, so that $p\mid m$. Write $m = p^r u$, where $p \nmid u$. Then $p\mid \sigma(u)$, and so $q^e\parallel u$ for some prime power $q^e$ with $p\mid \sigma(q^e)$. It follows that $n:=mP$ is an integer not exceeding $x$ divisible by $p^r q^e$. Hence, in this case our sum is at most  
\begin{align*} x\sum_{r\ge 1} \frac{1}{p^r} \sum_{\substack{q^e \le x\\p \mid \sigma(q^e)}}\frac{1}{q^e} &\ll \frac{x}{p} \Bigg(\sum_{\substack{q\le x \\ q\equiv -1\pmod{p}}}\frac{1}{q} + \sum_{\substack{q^e\le x,~e>1 \\ p\mid \sigma(q^e)}}\frac{1}{q^e} \Bigg) \\
&\ll \frac{x}{p}\Bigg(\frac{\log_2{x}}{p-1} + \frac{\log{p}}{p} + \sum_{\substack{m\text{ squarefull} \\ m > p/2}}\frac{1}{m}\Bigg),
\end{align*}
which is $o(x/p)$. 

Now assume $p\nmid a$. Fix $K> 2$ (which  later will  be taken large). We first bound the contribution to our sum from those cases where $m \le x^{1/K}$ or $m\ge x^{1-1/K}$. Since $\sigma(m)\equiv a\pmod{p}$ and $s(m)\equiv 0\pmod{p}$, we have that $m= \sigma(m)-s(m)\equiv a\pmod{p}$. Moreover, since $m>1$, we have $s(m)>0$, and so $\sigma(m) > s(m) \ge p$. Since $\sigma(m)\ll m\log_2{(3m)}$ (see, e.g., \cite[Theorem 323, p.\ 350]{HW08}), we deduce that $m \gg p/\log_2{p}$. It follows that the cases where $m\le x^{1/K}$ contribute
\begin{multline*} \ll \sum_{\substack{1 < m \le x^{1/K} \\ \sigma(m) \equiv a\pmod{p} \\ p \mid s(m)}} \pi(x/m) \ll  \frac{x}{\log{x}} \sum_{\substack{1 < m \le x^{1/K} \\ \sigma(m) \equiv a\pmod{p} \\ p \mid s(m)}} \frac{1}{m}  \\ \ll \frac{x}{\log{x}} \Bigg(\frac{\log_2{p}}{p} + \sum_{\substack{p < m \le x^{1/K} \\ m\equiv a\pmod{p}}}\frac{1}{m}\Bigg) \ll  \frac{x}{\log{x}} \Bigg(\frac{\log_2{p}}{p} + \frac{\log{x}}{pK}\Bigg), 
\end{multline*}
which is $o(x/p) + O(\frac{x}{pK})$.  If instead $m \ge x^{1-1/K}$, then $P \le x^{1/K}$. In that case it is convenient to count values of $m$ corresponding to a given $P$. We have that $m\equiv a\pmod{p}$, that $m\le x/P$, and that $m$ has no prime factors exceeding $P$. By the sieve, the number of possibilities for $m$ is $\ll \frac{x}{Pp} \prod_{P < q \le x/Pp}(1-1/q) \ll \frac{x}{p}\frac{\log{P}}{P\log{x}}$. (We assume here, and below, that  $x$ and $p/\log_2 x$ are large, in a manner allowed to depend on $K$, and we keep in mind that $p\le (\log{x})^A$.)
 Summing on $P \le x^{1/K}$, we see that the number of $n$ arising this way is $O(\frac{x}{pK})$.  

Now suppose that $x^{1/K} < m < x^{1-1/K}$. For each such $m$, the number of corresponding $P$ is at most $$ \pi(x/m) \ll \frac{Kx}{m\log{x}}.$$ We shall use this bound to justify several further assumptions on $m$. Since $p \mid s(m)$, we know that $m$ is not prime. Write $$m=m_0 P_1 P_2,$$ where $P_2 = P^{+}(m)$ and $P_1 = P^{+}(m/P_2)$. 

The number of $n:=mP$ corresponding to $m$ with $P_2 \le x^{1/K^3}$ is $$ \ll \frac{Kx}{\log{x}} \sum_{\substack{X^{1/K} < m \le x\\ m \equiv a\pmod{p} \\ P^{+}(m) \le X^{1/K^3}}} \frac{1}{m}.$$ By the sieve, for each $T \ge x^{1/K}$, the number of $m\le T$, $m\equiv a\pmod{p}$, with $P^{+}(m) \le x^{1/K^3}$ is $\ll\frac{T}{p} \prod_{x^{1/K^3} < q \le T/p} (1-1/q) \ll \frac{T}{pK^2}$. Hence, the sum of $1/m$ in the last display is $O(\frac{\log{x}}{pK^2})$, and the number of corresponding $n$ is  $O(\frac{x}{pK})$. Suppose $P_2 > x^{1/K^3}$ but $P_1 \le x^{1/K^3}$. Then $m=uP_2$ where $u:=m_0 P_1$ is such that $P^{+}(u) \le x^{1/K^3}$. Thus,
\begin{align*} \sum\frac{1}{m} \le \bigg(\sum_{\substack{P^{+}(u) \le x^{1/K^3} \\ p\nmid u}}  \frac{1}{u}\sum_{\substack{x^{1/K^3} < P_2\le x \\ P_2 \equiv u^{-1} a\pmod{p}}}\frac{1}{P_2}\bigg) \ll  \frac{\log{K}}{p} \sum_{\substack{P^{+}(u) \le x^{1/K^3} }} \frac{1}{u} \\ = \frac{\log{K}}{p} \prod_{q \le x^{1/K^3}} (1-1/q)^{-1} \ll \frac{\log{x}}{p} \cdot \frac{\log{K}}{K^3}.\end{align*}
Here the sum on $P_2$ has been estimated with the Brun--Titchmarsh inequality and partial summation (direct use of Lemma \ref{lem:PN} would give a slightly worse estimate). Hence, the number of corresponding $n$ is $O(\frac{\log{K}}{K^2} \frac{x}{p})$, which is $O(\frac{x}{pK})$.

Now suppose that $P_1, P_2 > x^{1/K^3}$. If $P_1=P_2$ or $P_1\mid m_0$, then $n=m_0 P_1 P_2 P$ is divisible by the square of a prime exceeding $x^{1/K^3}$. The number of such $n$ is $O(x^{1-1/K^3})$, which is $o(x/p)$. 

Thus, at the cost of $o(\frac{x}{p}) + O(\frac{x}{pK})$ exceptions, we may assume that $x^{1/K} < m < x^{1-1/K}$, that $P_2 > P_1 > P^{+}(m_0)$, and that $P_1 > x^{1/K^3}$. The congruence $\sigma(m)\equiv a\pmod{p}$ implies that $\sigma(m_0)$ is coprime to $p$, and that
\[ (P_1+1)(P_2+1)\equiv \sigma(m_0)^{-1} a \pmod{p}. \]
Also, $m \equiv a\pmod{p}$  implies that $p\nmid m_0$ and that
\[ P_1 P_2 \equiv m_0^{-1}a \pmod{p}. \]
For each $m_0$, the last two displayed congruences determine $O(1)$ possibilities for the pair of residue classes $(P_1 \bmod{p}, P_2\bmod{p})$. Moreover, for each pair $(u\bmod{p},v\bmod{p})$, the sum of $1/m$ taken over the corresponding values of $m=m_0 P_1 P_2$ does not exceed
\[ \sum_{m_0 \le x} \frac{1}{m_0} \sum_{\substack{x^{1/K^3} < P_1 \le x \\ P_1 \equiv u\pmod{p}}}\frac{1}{P_1} \sum_{\substack{x^{1/K^3} < P_2 \le x \\ P_2 \equiv v\pmod{p}}}\frac{1}{P_2} \ll \frac{(\log{K})^2}{p^2} \log{x}.  \]
Hence, the number of these remaining $n$ is 
\[ \ll \frac{K(\log{K})^2}{p} \frac{x}{p}, \]
which is $o(x/p)$.

Collecting the results of the last several paragraphs, we conclude that for each fixed $K$ the first of the sums in \eqref{eq:exactanalogy} is $ O(\frac{x}{pK})$, provided $x$ and $p/\log_2{x}$ are large enough (in terms of $K,A$). Since $K$ may be taken large, this first sum is $o(x/p)$. 

We have thus proved: Let $x$ and $p/\log_2{x}$ tend to infinity, with $p \le (\log{x})^A$ for a fixed $A>0$. Uniformly in the choice of $a\in \Z$, 
\begin{equation*}
\sum_{\substack{1 < n \le x\\ n~\text{composite} \\ s(n)\equiv a\pmod{p}}} 1 = \Bigg(\frac{1}{p-1} \sum_{\substack{m,P:~mP\le x\\ m>1,~P \ge P^{+}(m) \\ s(m)\not\equiv 0 \pmod{p}\\ \sigma(m)\not\equiv a\pmod{p}}} 1\Bigg)+o(x/p). \end{equation*}

Bounding the right-hand sum trivially yields the following analogue of Proposition  \ref{prop:philargecaseupper}.

\begin{prop}\label{prop:slargecaseupper} Fix $A>0$. The number of composite $n\le x$ for which $s(n)\equiv a\pmod{p}$ is $\lesssim x/p$, as $x,  \frac{p}{\log_2{x}} \to\infty$, uniformly in the choice of $a\in \Z$ and prime $p\le (\log{x})^A$.
\end{prop}

The analogue of Proposition \ref{prop:philargecaselower} can now be established.  We use in its proof that Proposition \ref{prop:philargecaseupper} still holds if $\phi$ is replaced by $\sigma$. In fact, our proof of Proposition \ref{prop:philargecaseupper} applies to $\sigma$ almost verbatim (a few ``$-$'' signs change to ``$+$'').

\begin{prop}\label{prop:slargecaselower} Fix $A>0$. The number of composite $n\le x$ for which $s(n)\equiv a\pmod{p}$ is $\gtrsim x/p$, as $x,  \frac{p}{\log_2{x}} \to\infty$, uniformly in the choice of $a\in \Z$ and prime $p\le (\log{x})^A$.
\end{prop}

\begin{proof} Since $\displaystyle\sum_{\substack{m,P:~mP\le x\\ m>1,~P \ge P^{+}(m)}} 1 = x-\pi(x) + O(1) \sim x$, it will suffice to show that both
\begin{equation}\label{eq:slowerA}\sum_{\substack{m,P:~mP\le x\\ m>1,~P \ge P^{+}(m) \\ \sigma(m)\equiv a\pmod{p}}} 1 = o(x) \end{equation}
and
\begin{equation}\label{eq:slowerB} \sum_{\substack{m,P:~mP\le x\\ m>1,~P \ge P^{+}(m) \\ s(m)\equiv 0 \pmod{p}}} 1 = o(x). \end{equation}

Let $L:=\exp(\sqrt{\log x})$. Imitating the argument  for \eqref{eq:phi2} in the proof of Proposition \ref{prop:philargecaselower}, we see that \eqref{eq:slowerA} and \eqref{eq:slowerB} follow if for all $T$ with $L \le T \le x$,
\[ \sum_{\substack{m\le T \\ \sigma(m)\equiv a\pmod{p}}} 1 \ll \frac{T}{p}, \quad \sum_{\substack{m\le T \\ s(m)\equiv 0\pmod{p}}} 1 \ll \frac{T}{p}. \]
The second estimate is a consequence of Proposition \ref{prop:slargecaseupper}, while when $p\nmid a$, the first  estimate follows from the $\sigma$-analogue of Proposition \ref{prop:philargecaseupper}. 

To prove \eqref{eq:slowerA} when $p\mid a$, we mimic the proof of \eqref{eq:phi1}. The sum on the left of \eqref{eq:slowerA} changes by $o(x)$ if we impose the additional constraint that $P\nmid m$. (In fact, our work above shows that the change is $O(x/L)$.) Then for the numbers $n=mP$ being counted here, $p \mid \sigma(m)\mid \sigma(mP)$, and so $n\le x$ is such that $p \mid \sigma(n)$. Lemma \ref{lem:sigmaasymp} now places $n$ in a set of size $o(x)$.
\end{proof}

Propositions \ref{prop:slargecaseupper} and \ref{prop:slargecaselower} complete the proof of Theorem \ref{thm:mains}.

\section{Proofs of Lemma \ref{lem:SDphi} and Lemma \ref{lem:SDs} by the method of Landau--Selberg--Delange}\label{sec:SD}
In this section, we supply the promised proofs of Lemmas \ref{lem:SDphi} and \ref{lem:SDs}, by the method of Landau--Selberg--Delange.  We use a recent formulation of that method due to Chang and Martin \cite{CM20}, which is based on Tenenbaum's treatment in \cite[Chapter II.5]{tenenbaum15} but (crucially for us) more explicit about the dependence on certain parameters.

\subsection{Setup} We follow \cite{CM20} in setting $\log^{+}{y} = \max\{0,\log{y}\}$, with the convention that $\log^{+}{0}=0$. We write complex numbers $s$ as $s=\sigma+i\tau$.\footnote{The distinction between $\sigma$ as the real part of a complex number and $\sigma$ as the sum-of-divisors function will be clear from context.} 

For a complex number $z$ and positive real numbers $c_0, \delta$, and $M$ satisfying $\delta \leq 1$, we say that the Dirichlet series $F(s)$ has \textsf{property $\mathcal P(z; c_0, \delta, M)$} if 
$$G(s; z):= F(s)\zeta(s)^{-z}$$
continues analytically for $\sigma \ge 1-c_0/(1+\log^{+}|\tau|)$, wherein it satisfies the bound
$$|G(s; z)| \leq M(1+|\tau|)^{1-\delta}.$$

For complex numbers $z$ and $w$ and for positive real numbers $c_0, \delta$, and $M$ satisfying $\delta \leq 1$, we say that a Dirichlet series $F(s) := \sum_{n=1}^\infty a_n n^{-s}$ has \textsf{type $\mathcal T(z, w; c_0, \delta, M)$} if it has property $\mathcal P(z; c_0, \delta, M)$ and there exists a sequence $\{b_n\}_{n=1}^\infty$ of nonnegative real numbers upper bounding the sizes of $\{a_n\}_{n=1}^\infty$ termwise (that is, satisfying $|a_n| \leq b_n$ for all positive integers $n$), such that 
the Dirichlet series $\sum_{n=1}^\infty b_nn^{-s}$ has property $\mathcal P(w; c_0, \delta, M)$.

The following is a special case  of Theorem A.13 in \cite{CM20}. Specifically, we take $A=1, N=0, \delta=1/2$ in that result.

\begin{prop}\label{prop:CM} Let $z,w$ be complex numbers with $|z|, |w|\le 1$. Let $c_0, M$ be positive real numbers with $c_0 \le 2/11$. Let $F(s)= \sum_{n=1}^{\infty} a_n/n^s$ be a Dirichlet series of type $\mathcal T(z, w; c_0, 1/2, M)$. Then, uniformly for $x\ge \exp(16/c_0)$,
we have 
\[ \sum_{n \le x} a_n = x (\log{x})^{z-1}\left(\frac{G(1;z)}{\Gamma(z)} + O(M R(x))\right),\]
where
\[ R(x) = c_0^{-3} \exp\left(-\frac{1}{6}\sqrt{\frac{1}{2}c_0\log{x}}\right) + \frac{1}{c_0\log{x}}. \]
\end{prop}
Here we have corrected some typos in \cite{CM20}; the expression for $R(x)$ there has an extra factor of $M$ throughout as well as an extra factor of $x$ in its first term.

\subsection{Proof of Lemma \ref{lem:SDs}} We prove Lemma \ref{lem:SDs} in detail; after that, it will suffice to sketch the (very similar) proof of Lemma \ref{lem:SDphi}. 

We will assume throughout the argument that $p\ge 3$. We do \emph{not} assume to start with that $p\to\infty$ or that $p$ and $x$ are related in size in a particular way; those assumptions of Lemma \ref{lem:SDs} will be introduced only at the conclusion of the argument.

By the orthogonality relations,
\begin{equation}\label{eq:orthocount} \sum_{\substack{n \le x \\ n \equiv u\pmod{p} \\ \sigma(n)\equiv v\pmod{p}}} 1 = \frac{1}{(p-1)^2} \sum_{\chi, \psi} \bar{\chi}{(u)} \bar{\psi}(v) \sum_{n \le x}\chi(n)\psi(\sigma(n)), \end{equation}
where the first right-hand sum is over all Dirichlet characters $\chi, \psi$ mod $p$. Let $\epsilon$ denote the trivial character mod $p$. Then
\[ \sum_{n\le x} \epsilon(n)\epsilon(\sigma(n)) = \sum_{\substack{n \le x \\ p \nmid n \\ p \nmid \sigma(n)}} 1, \]
whose behavior will be understood with Lemma \ref{lem:sigmaasymp}.  Now assume that $(\chi,\psi) \ne (\epsilon,\epsilon)$. In this case, we will estimate $\sum_{n\le x} \chi(n) \psi(\sigma(n))$ by an application of Proposition \ref{prop:CM}. 

Let
\[ F_{\chi,\psi}(s) = \sum_{n=1}^{\infty} \frac{\chi(n)\psi(\sigma(n))}{n^s}. \]
In the half plane $\Re(s)>1$,
\[ F_{\chi,\psi}(s) = \prod_{q} \left(1 + \frac{\chi(q)\psi(q+1)}{q^s} + \frac{\chi(q^2)\psi(q^2+q+1)}{q^{2s}} + \dots \right). \]
We can choose coefficients $a_{\rho}$, for each Dirichlet character $\rho$ mod $p$, in such a way that
\begin{equation}\label{eq:arhopurpose} \chi(n)\psi(n+1) = \sum_{\rho} a_{\rho} \rho(n) \end{equation}
for all $n$. Indeed, it is straightforward to check that this holds if we set
\begin{equation*} a_{\rho} = \frac{1}{p-1} \sum_{m \bmod{p}} (\chi\bar{\rho})(m) \psi(m+1). \end{equation*}

The sum on $m$ used to define $a_{\rho}$ has $p-2$ nonzero terms, and so trivially $|a_{\rho}| <1$. In fact, unless $\psi$ is trivial and $\rho=\chi$, we have $$|a_{\rho}|\le \sqrt{p}/(p-1).$$ This follows by recognizing $(p-1)a_{\rho}$ as --- up to sign --- a \textsf{Jacobi sum}.\footnote{In the theory of Jacobi sums, it is common to set $\epsilon(0)=1$. We are following instead the usual convention for Dirichlet characters according to which $\epsilon(0)=0$.}  See Theorem 1 on p.\ 93 and the Corollary on p.\ 94 of \cite{IR90}. This bound on $a_{\rho}$ can also be viewed as a consequence of  Weil's Riemann Hypothesis for curves (see, e.g., \cite[Corollary 2.3]{wan97} for a general character sum estimate along these lines).

We will show that $F_{\chi,\psi}(s)$ has property $\mathcal{P}(a_{\epsilon}; c_0, 1/2, M)$ for certain values $c_0 \le 2/11$ and $M \ge 1$. Since the coefficients of $F$ are termwise dominated by those of $\zeta(s)$, which has property $\mathcal{P}(1; c_0,1/2,M)$, it follows that $F_{\chi,\psi}(s)$ has type $\mathcal{T}(a_\epsilon, 1; c_0, 1/2,M)$. After obtaining estimates for $c_0$ and $M$, Proposition \ref{prop:CM} will yield a satisfactory estimate for $\sum_{n\le x} \chi(n) \psi(\sigma(n))$.

We set
\[ U_{\chi,\psi}(s) = F_{\chi,\psi}(s) \prod_{\rho} L(s,\rho)^{-a_{\rho}} \]
and observe that for $\Re(s)>1$,
\begin{equation*}
U_{\chi,\psi}(s) = \prod_{q} \left(\left(1 + \frac{\chi(q)\psi(q+1)}{q^s} + \frac{\chi(q^2)\psi(q^2+q+1)}{q^{2s}} + \dots \right)\prod_{\rho} \left(1-\frac{\rho(q)}{q^s}\right)^{a_{\rho}}\right). \end{equation*}
Notice that
\begin{multline*} \left(1 + \frac{\chi(q)\psi(q+1)}{q^s} + \frac{\chi(q^2)\psi(q^2+q+1)}{q^{2s}} + \dots \right)\left(1-\frac{\chi(q)\psi(q+1)}{q^s}\right)\\
= 1+ c_{2}/q^{2s} + c_3/q^{3s}+ \dots,
\end{multline*}
where the $c_j = c_j(q,\chi,\psi)$ are at most $2$ in absolute value. It follows that the function
\[ V_{\chi,\psi}(s) := \prod_{q} \left(\left(1 + \frac{\chi(q)\psi(q+1)}{q^s} + \frac{\chi(q^2)\psi(q^2+q+1)}{q^{2s}} + \dots \right)\left(1-\frac{\chi(q)\psi(q+1)}{q^s}\right)\right) \]
is holomorphic and bounded by an absolute constant for $\Re(s) \ge 0.99$ (say). 
For $\Re(s)>1$, 
\[ U_{\chi,\psi}(s) = 
V_{\chi,\psi}(s)  W_{\chi,\psi}(s), \]
where
\[ W_{\chi,\psi}(s) :=\prod_{q} \left(\bigg( \prod_{\rho} \bigg(1-\frac{\rho(q)}{q^s}\bigg)^{a_\rho}\bigg) \bigg(1-\frac{\chi(q)\psi(q+1)}{q^s}\bigg)^{-1}\right). \]
Recalling that the $a_{\rho}$ were selected to ensure \eqref{eq:arhopurpose}, we find that 
\[ \log W_{\chi,\psi}(s) = \sum_{q} \sum_{k\ge 2} \left(\frac{\chi(q)^k\psi(q+1)^k - \sum_{\rho} a_{\rho} \rho(q)^k}{kq^{ks}} \right). \]
This is holomorphic for $\Re(s) \ge 0.99$ and in this region we have
\[ |\log W_{\chi,\psi}(s)| \ll  1 + \sum_{\rho} |a_{\rho}|. \]
Moreover, since $|a_{\rho}|\le \sqrt{p}/(p-1)$ for all $\rho$, with at most one exception where $|a_{\rho}| < 1$, 
\begin{equation}\label{eq:arhosumbound} \sum_{\rho} |a_{\rho}| \ll \sqrt{p}. \end{equation}
We conclude that $U_{\chi,\psi}(s)$ is holomorphic for $\Re(s)\ge 0.99$ and that $|U_{\chi,\psi}(s)| \le \exp(O(\sqrt{p}))$ there.

Now
\begin{align*} F_{\chi,\psi}(s) &= U_{\chi,\psi}(s) \prod_{\rho} L(s,\rho)^{a_{\rho}}\\
&= \zeta(s)^{a_{\epsilon}} (1-1/p^s)^{a_{\epsilon}} U_{\chi,\psi}(s)\prod_{\rho \ne \epsilon} L(s,\rho)^{a_{\rho}}\\
&= \zeta(s)^{a_{\epsilon}} G_{\chi,\psi}(s),\qquad\text{where}\qquad G_{\chi,\psi}(s) := (1-1/p^s)^{a_{\epsilon}} U_{\chi,\psi}(s)\prod_{\rho \ne \epsilon} L(s,\rho)^{a_{\rho}}.
\end{align*}
The factor $(1-1/p^s)^{a_{\epsilon}}$ is holomorphic and absolutely bounded for $\Re(s) \ge 0.99$. It remains to understand the behavior of $\prod_{\rho \ne \epsilon} L(s,\rho)^{a_{\rho}}$. For this, we appeal to  \cite[Proposition 2.3]{CM20}. Below, $\log\log^{+}$ denotes the second iterate of $\log^{+}$.

\begin{lem}\label{lem:dircharbound} Let $m$ be an integer at least $3$. There is an effective constant $0 < \eta < 1/81$ such that for all $m\ge 3$ and all Dirichlet characters $\xi$ mod $m$, the function $L(s,\xi)$ has no zeros in the region 
\[ \sigma \ge 1-\frac{c_0}{1+\log^{+}|\tau|} \quad \text{with}\quad c_0 = \frac{\eta}{m^{1/2} (\log{m})^2} \]
and therein satisfies the bound
\[ |\log L(s,\xi)| \le 
\begin{cases}
\log\log^{+}(m|\tau|) + O(1) &\text{if $L(s,\xi)$ has no exceptional zero},\\
\frac{1}{2}\log{m} + 3\log\log^{+}(m|\tau|) + O(1)&\text{if $L(s,\xi)$ has an exceptional zero}.
\end{cases}\]
\end{lem}
We do not define ``exceptional zero'' here (see \cite{CM20}). It suffices for present purposes to note that for each $m$, there is at most one character $\xi$ mod $m$ for which $L(s,\xi)$ has an exceptional zero.

We take $$c_0 := \frac{\eta}{ p^{1/2}(\log{p})^{2}},$$ where $\eta$ is as in Lemma \ref{lem:dircharbound}. Then the product $\prod_{\rho \ne \epsilon} L(s,\rho)^{a_{\rho}}$ is nonzero and holomorphic for $\sigma \ge 1-c_0/(1+\log^{+}|\tau|)$, and in this same region, 
\begin{align*} |\log \prod_{\rho \ne \epsilon} L(s,\rho)^{a_{\rho}}|&\ll \log{p} + \sqrt{p} \log\log^{+}(p|\tau|) + O(\sqrt{p}) \\
&\ll \sqrt{p}(\log\log^{+}(p|\tau|) + 1).\end{align*}
(Here we used \eqref{eq:arhosumbound} and that at most one $\rho$ is exceptional.) 
Hence,
\[ \left|\prod_{\rho \ne \epsilon} L(s,\rho)^{a_{\rho}}\right| \le \exp(O(\sqrt{p})) \exp(O( \sqrt{p}\log\log^{+}(p|\tau|))). \]
It is a calculus exercise  to show that the right-hand side is at most $(C\sqrt{p})^{C\sqrt{p}} (1+|\tau|)^{1/2}$  for a certain absolute constant $C$. (Compare with the proof of \cite[Lemma 3.3]{CM20}.) 

Assembling our results, we find that $G_{\chi,\psi}(s)$ is holomorphic for $\sigma \ge 1-c_0/(1+\log^{+}|\tau|)$ and therein satifies 
\begin{equation}\label{eq:gchipsibound} |G_{\chi,\psi}(s)| \le M (1+|\tau|)^{1/2} \end{equation}
where $$ M:= (C'\sqrt{p})^{C'\sqrt{p}} $$ for a certain constant $C'$. Hence, $F_{\chi,\psi}(s)$ has property
$\mathcal{P}(a_{\epsilon}; c_0, 1/2, M)$.

From Proposition \ref{prop:CM}, we deduce that for all $x\ge \exp(\frac{16}{\eta} p^{1/2} (\log{p})^2)$, 
\[ \sum_{n \le x}\chi(n)\psi(\sigma(n)) = x(\log{x})^{a_\epsilon-1} \left(\frac{G_{\chi,\psi}(1)}{\Gamma(a_{\epsilon})} + O\left(M R(x)\right)\right).  \]

Since $(\chi,\psi)\ne (\epsilon,\epsilon)$, we have $|a_{\epsilon}| \le \sqrt{p}/(p-1)$. As $a_{\epsilon}$ is close to zero, $|1/\Gamma(a_{\epsilon})| \ll |a_{\epsilon}| \ll p^{-1/2}$. From \eqref{eq:gchipsibound}, $G_{\chi,\psi}(1) \ll M$. We have crudely that $R(x) \ll c_0^{-3} \ll p^2$. If we assume that $p> 10$, then $|a_{\epsilon}| < 2/5$, and we conclude that
\[ \bigg|\sum_{n\le x} \chi(n)\psi(\sigma(n))\bigg| \le x(\log{x})^{-3/5} \exp(O(\sqrt{p}\log p)). \]

Referring back to \eqref{eq:orthocount}, it is now straightforward to complete the proof of Lemma \ref{lem:SDphi}. We are assuming in Lemma \ref{lem:SDphi} that $x,p \to\infty$ with $p \le (\log_2{x})^{2-\delta}$. Under these assumptions, we certainly have (for large $x,p$) that $x\ge \exp(\frac{16}{\eta} p^{1/2} (\log{p})^2)$. Moreover (for large $x,p$), 
\[ \bigg|\sum_{n\le x} \chi(n)\psi(\sigma(n))\bigg| \le x(\log{x})^{-1/2}. \]
Therefore,
\[ \Bigg|\frac{1}{(p-1)^2} \sum_{\substack{\chi, \psi \\ (\chi,\psi)\ne (\epsilon,\epsilon)}} \bar{\chi}{(u)} \bar{\psi}(v) \sum_{n \le x}\chi(n)\psi(\sigma(n))\Bigg| \le x(\log{x})^{-1/2}. \]
On the other hand, Lemma \ref{lem:sigmaasymp} implies that
\[ \frac{1}{(p-1)^2}\sum_{n\le x} \epsilon(n)\epsilon(\sigma(n)) \sim \frac{1}{p^2} \frac{x}{(\log{x})^{1/(p-1)}}. \]
In this range, 
\[ \frac{x}{(\log{x})^{1/2}} = o\left(\frac{1}{p^2} \frac{x}{(\log{x})^{1/(p-1)}}\right). \]
We conclude that the number of $n\le x$ with $n\equiv u\pmod{p}$ and $\sigma(n)\equiv v\pmod{p}$ is $(1+o(1))\frac{x}{p^2 (\log{x})^{1/(p-1)}}$, as desired.

\section{Proof of Lemma \ref{lem:SDphi} (sketch)} The proof is similar to, but slightly simpler than, the above proof of Lemma \ref{lem:SDs}. We start by writing
\[ \sum_{\substack{n \le x \\ \phi(n)\equiv a\pmod{p}}} 1 = \frac{1}{p-1}\sum_{\chi} \bar\chi(a) \sum_{n \le x} \chi(\phi(n)). \]
For nontrivial $\chi$, we let $F_{\chi}(s) = \sum_{n=1}^{\infty} \chi
(\phi(n))n^{-s}$, and we define
\[ U_{\chi}(s) = F_{\chi}(s) \prod_{\rho} L(s,\rho)^{-a_{\rho}}, \]
where now each 
\[ a_{\rho} = \frac{1}{p-1} \sum_{m\bmod{p}}\bar{\rho}(m) \chi(m-1). \]
Here the $a_{\rho}$ have been chosen so that, for all $n\not\equiv 0\pmod{p}$, 
\[ \chi(n-1) = \sum_{\rho} a_{\rho} \rho(n). \]
Then $a_{\epsilon}=-\chi(-1)/(p-1)$ and $|a_{\rho}| \le \sqrt{p}/(p-1)$ for all $\rho\ne \epsilon$. Proceeding as before, one checks that $U_{\chi}(s)$ is holomorphic for $\Re(s) \ge 0.99$ and, in this same region, bounded in absolute value by $\exp(O(\sqrt{p}))$. From this, one deduces that $F_{\chi}(s)$ has type $\mathcal{T}(-\frac{\chi(-1)}{p-1},1;c_0,1/2,M)$ for $c_0:=\eta/(p^{1/2} (\log{p})^2)$, with $\eta$ as in Lemma \ref{lem:dircharbound}, and $M := (C\sqrt{p})^{C\sqrt{p}}$ for a certain absolute constant $C$. The rest of the argument is as above, using Lemma \ref{lem:phiasymp} in place of Lemma \ref{lem:sigmaasymp} at the appropriate spot.

\providecommand{\bysame}{\leavevmode\hbox to3em{\hrulefill}\thinspace}
\providecommand{\MR}{\relax\ifhmode\unskip\space\fi MR }
\providecommand{\MRhref}[2]{%
  \href{http://www.ams.org/mathscinet-getitem?mr=#1}{#2}
}
\providecommand{\href}[2]{#2}

\end{document}